\documentclass[12pt]{article}

\usepackage{amsthm}
\usepackage{amsmath}
\usepackage{amsfonts}
\usepackage{enumitem}
\usepackage{cite}
\usepackage{hyperref}
\usepackage{authblk}

\newtheorem{theorem}{Theorem}
\newtheorem{lemma}{Lemma}
\newtheorem{proposition}{Proposition}

\theoremstyle{definition}
\newtheorem{definition}{Definition}
\newtheorem{example}{Example}

\newcommand{\Z}[1]{\mathbb{Z}_{#1}}
\newcommand{\z}{\mathbb{Z}}

\begin{document}

\title{The Ideal of Vanishing Polynomials and the Ring of Polynomial Functions}
\author[1]{Matvey Borodin\thanks{Corresponding author: \href{mailto:matveyborodin1@gmail.com}{matveyborodin1@gmail.com}}}
\affil[1]{Brookline High School}
\author[2]{Ethan Liu}
\affil[2]{The Harker School}
\author[3]{Justin Zhang}
\affil[3]{Bergen County Academics}
\date{}
\maketitle

\begin{abstract}
Vanishing polynomials are polynomials over a ring which output $0$ for all elements in the ring. In this paper, we study the ideal of vanishing polynomials over specific types of rings, along with the closely related ring of polynomial functions. In particular, we provide several results on generating vanishing polynomials. We first analyze the ideal of vanishing polynomial over $\mathbb{Z}_n$, the ring of the integers modulo $n$. We then establish an isomorphism between the vanishing polynomials of a ring and the vanishing polynomials of the constituent rings in its decomposition. Lastly, we generalize our results to study the ideal of vanishing polynomials over arbitrary commutative rings. 
\end{abstract}

\section{Introduction}
Being a fundamental object in ring theory, polynomial rings over fields and integral domains have been the focus of a significant body of research; however, vanishing polynomials over rings with zero divisors have been less studied. In particular, it is widely known that a polynomial of degree $n$ can have at most $n$ roots over a field, but when zero-divisors are introduced this result may fail. More specifically, a polynomial that vanishes for all $x \in R$ must be of degree at least $|R|$ if $R$ is an integral domain, but if $R$ has zero divisors the degree of this polynomial can be significantly smaller. 

Previous works, such as \cite{SINGMASTER, Niven}, \cite{GREUEL}, and \cite{specker, GUHA} study the set of vanishing polynomials over integer and multivariable rings, respectively, and in a more closely related realm to this paper, \cite{specker, GUHA} studied the ring of polynomial functions over the integers modulo $n$. In this paper, we expand on this study of vanishing polynomials and their corresponding polynomial function over specific families of rings.

In Section~\ref{vanishing_polynomials}, we begin with definitions of essential terms followed by some background results that are used throughout the paper. In Section~\ref{vanishing_polynomials_over_znz}, we look at the generating set of the ideal of vanishing polynomials over $\Z{n}$, generalizing it to encompass results from various other works. In Section~\ref{structure_of_poly_functions}, we prove that the ring of vanishing polynomials over a product ring is isomorphic to the direct product of the rings of vanishing polynomials over each of the constituent rings. Similarly, the ring of polynomial functions over a product ring is isomorphic to the direct product of the rings of polynomial functions over each of the constituent rings. In Section~\ref{general_vanishing_polys} we find vanishing polynomials over arbitrary rings and apply these results to show that they give a complete classification for the integers modulo $n$. 

\section{Vanishing Polynomials} \label{vanishing_polynomials}

Throughout this paper, we assume rings to be commutative and with identity unless otherwise specified. Let $R$ be a ring and let $R[x]$ denote the ring of polynomials with coefficients in $R$. We study objects in $R[x]$ known as \emph{vanishing polynomials}.

\begin{definition}
A \emph{polynomial} $F(x)$ in a polynomial ring $R[x]$ is a formal sum $$a_{n}x^{n}+a_{n-1}x^{n-1}+\cdots+a_{1}x+a_{0}$$ for some nonnegative integer $n$, where each $a_{i} \in R$ and $x$ is an indeterminate.
\end{definition}

\begin{definition}
A \emph{vanishing polynomial} $F(x) \in R[x]$ is a polynomial such that $F(a) = 0$ for all $a \in R$. By definition, $0$ itself is a vanishing polynomial.
\end{definition}

The set of vanishing polynomials over a ring $R[x]$ is known to form an ideal. Thus, taking $I$ to be the ideal of vanishing polynomials, it makes sense to consider the quotient $R[x]/I$. To understand the significance of this operation we must first consider the distinction between a \emph{polynomial} and a \emph{polynomial function}. 

\begin{definition}
A \emph{polynomial function} $f: R \to R$ is a function on $R$ for which there exists a polynomial $F(x) \in R[x]$ such that $f(r)=F(r)$ for all $r \in R$.
\end{definition}

When we need to distinguish between the two, we refer to the polynomial with an uppercase letter, such as $F(x)$, and the corresponding polynomial function with a lowercase letter, such as $f(x)$. Note that each polynomial $F(x)$ corresponds to a unique polynomial function $f(x)$, but as we will see later, any given polynomial function $f(x)$ must correspond to an infinite number of polynomials if there is a nonzero vanishing polynomial. The proposition below is a statement found in \cite{gilmer1999ideal} connecting polynomials and polynomial functions. 

\begin{proposition}
\label{quotient}
The quotient $R[x] / I$, where $I$ denotes the ideal of vanishing polynomials over $R[x]$, is the ring of polynomial functions $f: R \to R$. 
\end{proposition}

\section{Vanishing Polynomials Over $\Z{n}$} \label{vanishing_polynomials_over_znz}

It was shown in \cite{SINGMASTER} that any element of the ideal of vanishing polynomials over $\Z{n}$ is of the form 
\begin{equation}\label{van_poly}
G(x) = F(x)B_s(x) + \sum_{k=0}^{s - 1} a_k \cdot\frac{n}{\gcd(k!, n)} \cdot B_k(x)
\end{equation}
where $B_k(x) = (x+1)(x+2)\ldots (x+k)$ with $B_0(x) = 1$, and $s$ is the smallest integer such that $n \mid s!$, known as the Smarandache or Kempner function. $F(x)$ is a polynomial which is uniquely defined based on $G(x)$, and $a_k$'s are integers also uniquely defined in the range $0 \leq a_k < \gcd(k!, n)$. Similar results were presented in \cite{GREUEL, specker, Niven} with slightly different definitions of $B_k(x)$, which we discuss later.

We propose a new method of finding vanishing polynomials with integer coefficients over $\Z{n}$, inspired by \cite{newman_1982}. It is well known that any integer-valued polynomial can be uniquely written as a linear sum with integer coefficients of functions of the form $\binom{x}{k} = x(x-1)(x-2)\ldots(x-k+1)/k!$. If we want a polynomial to vanish modulo $n$, it must evaluate to an integer multiple of $n$ for all integer inputs, so it must be $n$ times an integer-valued polynomial. Therefore, any vanishing polynomial $F(x)$ corresponds to an integer-valued polynomial $G(x) = F(x)/n$. Conversely, for an integer-valued polynomial $G(x)$ to correspond to a polynomial $F(x) = nG(x)$, all resulting coefficients in $F(x)$ must be integers. 

As mentioned previously, every such $G(x)$ can be uniquely represented as a sum $G(x) = \sum_{k = 1}^{m} c_k \binom{x}{k}$ so every vanishing polynomial $F(x)$ over $\Z{n}$ can be uniquely represented as 
$$
F(x) = \sum_{k = 1}^{m} n c_k \binom{x}{k}
$$
where $c_k$ and $m$ are integer values. Now we classify which sums of this form yield valid polynomials. First, we will show that each element of the sum must have integer coefficients.

\begin{proposition}
If any term in the summation $\sum_{k=1}^{m} nc_k \binom{x}{k}$ has a non-integer coefficient then the resulting polynomial cannot have integer coefficients. 
\end{proposition}
\begin{proof}
Assume for contradiction that we can construct a polynomial of this form where one of the terms has a non-integer coefficient. If any term $nc_i \binom{x}{i} = \frac{nc_i}{i!}\cdot x(x-1)(x-2)\ldots(x-i+1)$ has a non-integer coefficient this means $i!$ does not divide $nc_i$ and therefore the leading coefficient of this term, that is the coefficient on $x^i$, is also non-integer. Thus, if the sum is to have all integer coefficients, there must be a term of greater degree, say $j$, that has a non-integer coefficient on $x^i$. As before, this implies that the leading coefficient of this new term, that is the coefficient on $x^j$, is also non-integer. Inductively, we conclude that for every term with non-integer coefficients, there must be a term of greater degree with non-integer coefficients and therefore there is no term of maximal degree. We have reached a contradiction. 
\end{proof}

    For an element of the form $nc_k\cdot \binom{x}{k}$ to have integer coefficients, we must have $k! \mid nc_k$. The smallest such $c_k$ is $k! /\gcd(n, k!)$ and any greater $c_k$ would be a multiple of this, so any valid $c_k$ can be written as $a_k \cdot (k! /\gcd(n, k!))$ for an arbitrary integer $0 \leq a_k < \gcd(n, k!)$. Note that any $a_k$ outside this range gives a polynomial equivalent to having an $a_k$ in this range by reduction modulo $n$. If we define $s$ to be the smallest integer such that $n \mid s!$, any polynomial $na_k\cdot (k!/\gcd(k!, n))\cdot \binom{x}{k}$ where $k \geq s$ is a polynomial multiple of $\binom{x}{s}$, therefore we have arrived at exactly the formulation in Equation~\ref{van_poly} except with $B_k(x)$ defined as $k!\cdot \binom{x}{k}$. 

An immediate consequence of this formula is that the generating set for the ideal of vanishing polynomials over $\Z{n}$ is $$\left\{\frac{n}{\gcd(k!, n)} \cdot B_k(x) \mid k \in \z_{\geq 0}\right\}$$ for either definition of $B_k(x)$. Note that if $k$ is less than the smallest prime divisor of $n$, we get the zero polynomial.

From this formulation, we can immediately find the degree of the minimal degree monic vanishing polynomial and minimal degree non-monic vanishing polynomial, which would be $s$ and the smallest factor of $n$, respectively. Note that the minimal degree non-monic polynomial must be unique up to multiplication by a constant since the generating set only contains a single nonzero polynomial of that degree or lower. 

There exists a far simpler description of this minimal-degree polynomial, namely $\frac{n}{p_1}\cdot x(x^{p_1 - 1} - 1)$ where $p_1$ is the smallest prime divisor of $n$. Fermat's little theorem states that $x^{p_1 - 1} - 1 = 0$ modulo $p_1$ for any $x$ not divisible by $p_1$. To take care of the $p_1 \mid x$ case we multiply the whole polynomial by $x$. Thus, the polynomial $x(x^{p_1 - 1} - 1)$ is a multiple of $p_1$ for any input, and when multiplied by the leading coefficient $n/p_1$ it vanishes over $\Z{n}$. Since the leading coefficient and degree matches the other formulation we have for the minimal degree monic polynomial they must be the same polynomial over $\Z{n}$, which is an interesting identity.

While the generating set we found above is convenient in many ways because it uses linear combinations, if we consider it from the point of view of a generating set of an ideal, that is allow any coefficients in the ring, many of the elements become redundant. In particular, if we have two polynomials $a\cdot \binom{x}{i}\cdot i!$ and $a\cdot \binom{x}{j}\cdot j!$ for some integer $a$, and $i < j$ then the second polynomial is a polynomial multiple of the first and therefore redundant in a generating set. Thus, to minimize our generating set we can remove any polynomials $k!\cdot (n/\gcd(k!, n))\cdot \binom{x}{k}$ for which $k$ is not the minimal integer which gives the same value of $\gcd(k!, n)$. 


Let us return to the varying definitions of $B_k(x)$ across multiple papers. In fact, $B_k(x)$ can be replaced with any sequence of terms such that the product is divisible by $\gcd(n, k!)$ and the set of generators will remain valid (for instance, $B_k(x) = (x+i)(x+i+1)\ldots(x+i+k-1)$ for any integer $i$). The fact that a polynomial that uses any such $B_k(x)$ is vanishing is easy to prove. The coefficient on $B_k(x)$ must be a multiple of $n/\gcd(n, k!)$, thus the product is a multiple of $n$ and the polynomial is vanishing. Furthermore, any polynomial $F(x)$ with degree $d$ and integer coefficients can be uniquely decomposed as
\begin{equation}\label{van_poly_decomposition}
F(x) = \sum_{k=0}^{d} b_k B_k(x)
\end{equation}
using polynomial division as long as each $B_k(x)$ has degree $k$ and is monic. It remains to show that every vanishing polynomial decomposed in this way has coefficients conforming with the required conditions, namely showing that $b_k = a_k\cdot n/\gcd(n, k!)$ for an integer $a_k$. We show this by polynomial division. Let our starting vanishing polynomial $F(x)$ have degree $d$. By Equation~\ref{van_poly}, we know the leading coefficient must be of the form $a_d\cdot \gcd(n, d!)$ for some integer $a_d$. When we take $F(x) - a_d\cdot \gcd(n, d!)\cdot B_d(x)$, we have the difference of two vanishing polynomials, giving another vanishing polynomial $F_1(x)$ with degree $d_1 < d$. Applying Equation~\ref{van_poly} again, the leading coefficient of $F_1(x)$ must be of the form $a_{d_1} \cdot \gcd(n, d_1!)$ so we can repeat the procedure to get a new vanishing polynomial $F_2(x)$ with degree $d_2 < d_1$. Since the initial degree $d$ is finite, the process must terminate with $d_m = 0$ at step $m$, showing that $b_k = a_k \cdot n/\gcd(n, k!)$ in Equation~\ref{van_poly_decomposition} for a vanishing polynomial $F(x)$, as desired.

\section{Vanishing Polynomials Over Product Rings} \label{structure_of_poly_functions}

The results in Section \ref{vanishing_polynomials_over_znz} can be generalized to direct products of multiple rings as follows.

Let $k \geq 2$ be an arbitrary positive integer, and 
let $R_{1}, \ldots, R_{k}$ be finite commutative rings. 
Let $R \cong R_{1} \times \cdots \times R_{k}$, and 
let $I$ be the vanishing polynomial ideal of $R[x]$, $I_{1}$ be the ideal of vanishing polynomials of $R_{1}[x]$, and so on. 
For any $R_{i}$, let $\pi_{R_{i}} : R \to R_{i}$ be the canonical projection mapping onto $R_{i}$.

First, we establish an isomorphism that maps vanishing polynomials in $R[x]$ to corresponding tuples of vanishing polynomials in $R_{1}[x], \ldots, R_{k}[x]$ and vice versa. The following well known lemma is found in \cite{ahmed2020commutative}, but we include the following proof as the notions developed in it are used in other results throughout this section.

\begin{lemma} \label{polynomial_product}
Given $R \cong R_{1} \times \cdots \times R_{k}$, we have
$$R[x] \cong R_{1}[x] \times \cdots \times R_{k}[x].$$
\end{lemma}

\begin{proof}
We define a mapping $\phi$ from $R[x]$ to $R_{1}[x] \times \cdots \times R_{k}[x]$ which breaks down a polynomial $P(x) \in R[x]$ of degree $m$ into a $k$-tuple of polynomials by splitting up the components in the coefficients. More specifically, let
$$P(x)=\sum_{i=0}^{m}a_{i}x^{i}.$$
Then,
$$\phi(P(x)) = \left(\sum_{i=0}^{m}\pi_{R_{1}}(a_{i})x^{i}
,\sum_{i=0}^{m}\pi_{R_{2}}(a_{i})x^{i}, 
\ldots, \sum_{i=0}^{m}\pi_{R_{k}}(a_{i})x^{i}\right).$$
Notice that $\phi$ is a homomorphism because for any arbitrary polynomials $P(x),Q(x)$ in $R[x]$, $\phi(P(x)+Q(x))$ is the componentwise sum of $\phi(P(x))$ and $\phi(Q(x))$, and $\phi(P(x) Q(x))$ is the componentwise product of $\phi(P(x))$ and $\phi(Q(x))$. In addition, we define $\phi^{-1}$ from $R_{1}[x] \times \cdots \times R_{k}[x]$ to $R[x]$ mapping $k$ polynomials $P_{1}(x) \in R_{1}[x], \ldots, P_{k}(x) \in R_{k}[x]$ with highest degree $m$ to a polynomial $P(x) \in R[x]$. For any integer $j$ from 1 to $k$, let
$$P_{j}=\sum_{i=1}^{m}a_{i,j}x^{i}.$$
Then,
$$\phi^{-1}((P_{1}(x), \ldots, P_{k}(x))) = \sum_{i=1}^{m}(a_{i,1},a_{i,2}, \ldots, a_{i,k})x^{i}.$$

We now note that for any polynomial $P(x) \in R[x]$, $\phi(\phi^{-1}(P(x)))=P(x)$
Therefore $\phi$ is invertible, so it is an isomorphism.
\end{proof}

\begin{theorem}
The ideal $I$ of vanishing polynomials in $R[x]$ is isomorphic to the direct products of the ideals $I_{1}, \ldots, I_{k}$ of vanishing polynomials in $R_{1}[x], \ldots, R_{k}[x]$, respectively. In other words,
$$I \cong I_{1} \times \cdots \times I_{k}.$$
\end{theorem} 

\begin{proof} Consider a vanishing polynomial $P(x)$ in $R[x]$. Let $P_{i}(x) = \pi_{R_{i}}(\phi(P(x)))$ be the canonical projection of $P(x)$ onto $R_{i}[x]$ for all integers $i$ from 1 to $k$, as defined prior. Then, each $P_{i}(x)$ must be a vanishing polynomial in $R_{i}[x]$. Therefore any vanishing polynomial in $R$ has vanishing polynomials in $R_{1}[x], \ldots, R_{k}[x]$. 

Going in the other direction, suppose we have a $k$-tuple of vanishing polynomials $(P_{1}(x), \ldots, P_{k}(x))$, where $P_i(x)$ is in polynomial ring $R_i[x]$ for $1\leq i\leq k$. It then follows that there exists a polynomial $P(x)$ in $R[x]$ such that $P(x)=\phi^{-1}((P_{1}(x), \ldots, P_{k}(x)))$. It follows that $P(x)$ is a vanishing polynomial. Therefore, if we have a $k$-tuple of vanishing polynomials in each of $R_{1}[x], \ldots, R_{k}[x]$, we can construct a vanishing polynomial in $R[x]$. Thus, $I$ is isomorphic to $I_{1} \times \cdots \times I_{k}$.
\end{proof}

Now we consider an isomorphism between the rings of polynomial functions over $R$ and over $R_1\times R_2 \times \cdots \times R_k$. 

\begin{theorem} Let $R$ be the direct product of $k$ rings $R_{1}, \ldots, R_{k}$. Then, the ring of polynomial functions on $R$ is isomorphic to the direct product of the rings of polynomial functions on $R_{1}, \ldots, R_{k}$. In other words,
$$R[x]/I \cong R_{1}[x]/I_{1} \times \cdots \times R_{k}[x]/I_{k}.$$
\end{theorem}

\begin{proof} We have already established that $R[x]/I$ is the ring of polynomial functions over $R$. Similarly, $R_{1}[x]/I_{1}$ is the ring of polynomial functions over $R_{1}$, and so on. 

Now, consider a polynomial function over $R$. Each polynomial function $f$ is in essence a mapping of every possible input $k$-tuple to an output $k$-tuple. If we consider only the $i$th component of each $k$-tuple, we construct a mapping $\beta_{i}$ which maps a polynomial function $f \in R[x]/I$ to a function $f_{i}$ over $R_{i}$ such that for any $r \in R$, $$\pi_{R_{i}}(f(r))=f_{i}(\pi_{R_{i}}(r)).$$
Notice that given any two polynomial functions $f \mapsto f_{i}$ and $g \mapsto g_{i}$, we have
$$\pi_{R_{i}}(f(r))+\pi_{R_{i}}(g(r))=f_{i}(\pi_{R_{i}}(r))+g_{i}(\pi_{R_{i}}(r)).$$ Therefore, by the definition $\pi_{R_i}$,
$$\pi_{R_{i}}(f(r)+g(r))=(f_{i}+g_{i})(\pi_{R_{i}}(r)).$$
This means that
$$\beta_i(f+g) = \beta_i(f) + \beta_i(g).$$
A similar argument holds over multiplication, so $\beta_i$ is a homomorphism.

Now we must show that $f_{i}$ is a polynomial function (note that not all functions can be expressed as polynomials). Let $F$ be a polynomial in $R[x]$ that corresponds to the polynomial function $f$. Then, take only the $i$-th component of each coefficient in $F$. We now have a polynomial $F_{i}$ in $R_{i}[x]$ that evaluates to $f_{i}$. Therefore $f_{i}$ is a polynomial function, so every $f$ over $R$ can be split into $f_{1}, \ldots, f_{k}$ over $R_{1}, \ldots, R_{k}$, respectively.

On the other hand, suppose we have a tuple of $k$ polynomial functions $f_{1}, \ldots, f_{k}$ over $R_{1}, \ldots, R_{k}$, respectively. We  construct a function $f$ over $R$ that maps any input tuple $(a_{1}, \ldots, a_{k})$ to an output tuple $(f_{1}(a_{1}), \ldots, f_{k}(a_{k}))$.

We now show that $f$ is a polynomial function. Let $F_{1}, \ldots, F_{k}$ be polynomials in $R_{1}[x], \ldots, R_{k}[x]$ which evaluate to $f_{1}, \ldots, f_{k}$, respectively, and $G_{1}, \ldots, G_{k}$ their extensions in $R$. Let $F = (1, \ldots, 0)G_{1} + \cdots + (0, \ldots, 1)G_{k}$, where every $G_{i}$ is multiplied by an $k$-tuple whose $i$th component is $1$ and whose other components are $0$. This polynomial $F$ evaluates to the function $f$. Therefore for all sets of $k$ polynomial functions $f_{1}, \ldots, f_{k}$ over $R_{1}, \ldots, R_{k}$, respectively, we have a polynomial function $f$ over $R$. Thus, the ring of polynomial functions over $R$ is isomorphic to the direct product of the rings of polynomial functions over $R_{1}, \ldots, R_{k}$, respectively.
\end{proof}

\section{Vanishing Polynomials Over General Rings} \label{general_vanishing_polys}

We can apply a generalization of the method used in Section~\ref{vanishing_polynomials_over_znz} to find vanishing polynomials over an arbitrary commutative ring $R$ with identity. Given an element $r \in R$, we denote the ideal generated by $r$ as $(r) = \{qr \mid q \in R\}$. We begin with a definition that simplifies the construction.

\begin{definition}
Given an element $y \in R$ such that the quotient $R/(y)$ partitions $R$ into a finite number of equivalence classes of the form $a_1 + (y), a_2 + (y), \ldots, a_n + (y)$, let $F_y(x) = (x-a_1)(x-a_2)\ldots(x-a_n)$. If $R/(y)$ is not finite we say $F_y(x)$ is not defined.
\end{definition}

We claim that $F_y(x)$ evaluates to a multiple of $y$ for any input $x \in R$. Note that $x$ must fall into some equivalence class $a_i + (y)$, so $x-a_i \in (y)$. Since $F_y(x)$ contains the factor $(x - a_i) \in (y)$ and $(y)$ is an ideal, $F_y(x) \in (y)$ so $F_y(x)$ is a multiple of $y$.

\begin{theorem} \label{general_vanishing_poly}
Consider a set of zero divisors $\{y_1, y_2, \ldots, y_n\}$ which satisfies $y_1y_2\ldots y_n = 0$. Let $N \subset \mathbb{N}$ be an indexing set which can contain $i$ if $F_{y_i}(x)$ is defined (but does not necessarily contain $i$ if $F_{y_i}(x)$ is defined). Let $M \subset \mathbb{N}$ contain all integers $1\leq j \leq n$ such that $j \notin N$. Then the polynomial
$$
G(x) = \left(\prod_{j \in M} y_j\right) \left( \prod_{i \in N} F_{y_i}(x)\right)
$$
is vanishing.
\end{theorem}
\begin{proof}
This follows directly from the definition of $F_y(x)$. For any input $x$, $G(x)$ evaluates to a multiple of $y_1y_2\ldots y_n = 0$ so it is vanishing.
\end{proof}

Observe that if all quotients $R/(y_1), R/(y_2),\ldots, R/(y_n)$ are finite, we can take the subset of $y_i$'s to be the whole set in which case we get a monic polynomial. Also, if $R$ is finite, setting $y_1 = 0$ gives us the monic vanishing polynomial $F_0(x)$ whose degree is $|R|$ (such as for a finite field which has no other zero divisors). 

\begin{example}
This method allows us to find vanishing polynomials not only for finite rings but also for infinite ones. For instance, in the ring $\prod_{n=1}^{\infty} \Z{2}$, consider the pair of zero divisors $(0,1,1,1,1,\ldots)$ and $(1,0,0,0,0,\ldots)$. Since the quotient of this ring by the ideal generated by $(0,1,1,1,1,\ldots)$ is finite (it is isomorphic to $\Z{2}$), there exists a degree 2 vanishing polynomial for this ring. Furthermore, since we can pick any representatives from the cosets in the quotient there exists an infinite number of degree 2 vanishing polynomials for this ring. 
\end{example}

Note that a polynomial generated in this way will often have a lot of redundant terms since if we have two finite quotients $R/(y_i)$ and $R/(y_j)$, there will often be pairs of factors $(x-c) = (x-d)$ with $c$ chosen from one of the cosets of $R/(y_i)$ and $d$ chosen from one of the cosets of $R/(y_k)$. To remove such redundancies we can create groups of identical factors with the additional condition that factors corresponding to quotients $R/(y_i)$ and $R/(y_j)$ can only be grouped together if being a multiple of both $y_i$ and $y_j$ implies being a multiple of $y_i y_j$ (analogous to $y_i$ and $y_j$ being coprime over the integers). Once such a grouping is carried out in a way minimizing the number of groups (including picking representatives from cosets to facilitate this process), we simply keep a single factor from each of the groups. 

\begin{example}
Consider $(x-1)(x-2)\cdot (x-1)(x-2) \cdot (x-1)(x-2)(x-3)$, a polynomial over the ring $\Z{12}$ using the zero divisors $2 \cdot 2 \cdot 3 = 0$. We cannot group terms corresponding to the two instances of $\Z{12}/(2)$ but we can group terms corresponding to $\Z{12}/(2)$ with terms corresponding to $\Z{12}/(3)$. We can also replace one of the terms $(x-1)$ with $(x-3)$ as it belongs to the same equivalence class modulo 2. Thus, we get the grouping pairs $(x-1)(x-1)\cdot(x-2)(x-2)\cdot (x-3)(x-3)\cdot (x-2)$, giving the final reduced polynomial $(x-1)(x-2)(x-3)(x-2)$. 
\end{example}

\begin{proposition}
If $R = \Z{n}$, this description is sufficient to classify all vanishing polynomials when we take $y_1\ldots y_k = n$ to be the prime factorization of $n$.
\end{proposition}

\begin{proof}
This is essentially the proof found at the end of Section~\ref{vanishing_polynomials_over_znz}. Let us define $N$ such that $\prod_{i \in N} y_i = \gcd(n, k!)$. Then $G(x)$ as defined in Theorem~\ref{general_vanishing_poly} behaves identically to $\frac{n}{\gcd(k!, n)} B_k(x)$ in Equation~\ref{van_poly}. In particular, note that we get the same leading coefficients.

Thus, to completely reuse the proof in Section~\ref{vanishing_polynomials_over_znz} it simply remains to show that this $G(x)$ has degree less than or equal to $k$ after removing duplicate terms. Let us define $z_1^{e_1}z_2^{e_2}\ldots z_l^{e_l} = \gcd(n, k!)$ to be the prime factorization of $\gcd(n, k!)$. Using the method above we can chose representatives to get $F_{z_i}(x)F_{z_i}(x)\ldots F_{z_i}(x) = (x-1)(x-2)\ldots(x-z_ie_i)$ for each $z_i^{e_i}$. After performing the grouping as described, we simply get $(x-1)(x-2)\ldots(x-z_je_j)$ for whichever $z_je_j$ was maximal (since all $z_i$ are coprime). Now, if $k!$ has the factor $z_i^{e_i}$, $k$ is at least $z_i e_i$ so the polynomial we get has degree less than or equal to $k$. From this point, the proof is identical, except with $\frac{n}{\gcd(k!, n)} B_k(x)$ replaced by an appropriate $G(x)$ for each value of $k$.
\end{proof}

This is not the case in general. For instance, the polynomial $x(x+(1,1,1,1,\ldots))$ is vanishing over $\prod_{n=1}^{\infty} \Z{2}$ but does not correspond to any pair of zero divisors as described previously. This gap, however, can be easily fixed by noticing that if a ring $R$ can be represented as a direct product of a ring $S$ with itself and some other ring $T$, namely $R = S\times S \times T$, a vanishing polynomial for the ring $S\times T$ can be trivially generalized to a vanishing polynomial for the ring $R$ by replacing any terms $(s, t)$ with $(s, s, t)$. In particular, doing this infinitely for $S = \Z{2}$, we get the polynomial $x(x+(1,1,1,1,\ldots))$ over $\prod_{n=1}^{\infty} \Z{2}$. 

\section{Acknowledgements}

We would like to thank our research advisor, Prof. James Barker Coykendall, for mentoring us throughout this research process. We would also like to thank Dr. Felix Gotti and Dr. Tanya Khovanova for reviewing the paper and providing many helpful comments for revision. Finally, we would like to express our gratitude to Dr. Pavel Etingof, Dr. Slava Gerovitch, and the MIT PRIMES program for making this project possible.

\bibliography{refs}{}

\begin{thebibliography}{1}

\bibitem{ahmed2020commutative}
Malik~Tusif Ahmed, Tiberiu Dumitrescu, and M~Azeem Khadam.
\newblock Commutative rings with absorbing factorization.
\newblock {\em Communications in Algebra}, 48(12):5067--5075, 2020.

\bibitem{gilmer1999ideal}
R.~Gilmer.
\newblock The ideal of polynomials vanishing on a commutative ring.
\newblock {\em Proceedings of the American Mathematical Society}, 127(5):1265--1267, 1999.

\bibitem{GREUEL}
G.-M. Greuel, F.~Seelisch, and O.~Wienand.
\newblock The gröbner basis of the ideal of vanishing polynomials.
\newblock {\em Journal of Symbolic Computation}, 46(5):561--570, 2011.
\newblock Groebner Bases and Applications.

\bibitem{GUHA}
A.~Guha and A.~Dukkipati.
\newblock A faster algorithm for testing polynomial representability of functions over finite integer rings.
\newblock {\em Theoretical Computer Science}, 579:88--99, 2015.

\bibitem{newman_1982}
D.J. Newman.
\newblock {\em A Problem Seminar}.
\newblock Springer New York, 1982.

\bibitem{Niven}
I.M. Niven and L.J. Warren.
\newblock A generalization of fermat’s theorem.
\newblock 1957.

\bibitem{SINGMASTER}
D.~Singmaster.
\newblock On polynomial functions (mod m).
\newblock {\em Journal of Number Theory}, 6(5):345--352, 1974.

\bibitem{specker}
E.~Specker, N.~Hungerbühler, and M.~Wasem.
\newblock The ring of polyfunctions over $\mathbb {Z}/n\mathbb {Z}$.
\newblock 2021.

\end{thebibliography}
\bibliographystyle{plain}

\end{document}